\UseRawInputEncoding
\documentclass[12pt]{amsart}
\usepackage{cases}
\usepackage{amscd}
\usepackage{amsmath}
\usepackage{amsthm}
\usepackage{amssymb}
\usepackage{mathrsfs}
\usepackage{amsfonts}
\usepackage{color}

\usepackage{pifont}

\usepackage{upgreek}
\usepackage{bm}

\usepackage{indentfirst, latexsym, bm, amsmath, eufrak, amsthm}
\usepackage{amsmath}
\usepackage{amsthm}
\usepackage{amssymb}
\usepackage{mathrsfs}
\usepackage{amsfonts}

\usepackage{pifont}

\usepackage{upgreek}
\usepackage{bm}
\newcommand*{\tr}{\mathrm{tr}}
\numberwithin{equation}{section}
\newtheorem{theo}{Theorem} 

\newtheorem{lem}{Lemma}
\newtheorem{mcor}{Corollary}
\newtheorem{remark}{Remark}

\newcommand*{\D}[1]{\ensuremath{\nabla^{#1}}}

\begin{document}
\title[A Schur type lemma for the Mean Berwald curvature]{A Schur type lemma for the Mean Berwald curvature in Finsler geometry}

\author[Ming Li]{Ming Li$^1$}



\address{Ming Li: Mathematical Science Research Center,
Chongqing University of Technology,
Chongqing 400054, P. R. China }

\email{mingli@cqut.edu.cn}


\thanks{$^1$~Partially supported by NSFC (Grant No. 11871126) }

\maketitle

\begin{abstract}
In this short paper, we study a symmetric covariant tensor in Finsler geometry, which is called the mean Berwald curvature. We first investigate the geometry of the fibres as the submanifolds of the tangent sphere bundle on a Finsler manifold. Then we prove that if the mean Berwald curvature is isotropic along fibres, then the Berwald scalar curvature is constant along fibres.
\end{abstract}


\section*{Introduction}

Let $(M,F)$ be a Finsler manifold of $\dim M=n$. Let $SM=\{(x,y)\in TM|F(x,y)=1\}$ be the unit tangent sphere bundle of $M$ with the natural projection $\pi:SM\to M$. The tangent bundle $T(SM)$ admits a natural horizontal subbundle $H(SM)$ and a Sasaki-type metric $g^{T(SM)}$. The Hilbert form $\omega=F_{y^i}dx^i$ defines a contact structure on $SM$. The Reeb field is just the spray $\mathbf{G}$.

Let $\mathcal{D}$ be the contact distribution $\{\omega=0\}$. There is an almost complex structure $J$ on $\mathcal{D}$.
$J$ extends to be an endomorphism  $T(SM)$ by defining $J(\mathbf{G})=0$. It is clear that $(SM,J,\mathbf{G},\omega,g^{T(SM)})$ gives rise to a contact metric structure.

Let $\mathscr{F}:=V(SM)$ be the integrable distribution given by the tangent spaces of the fibres of $SM$.  Let $p:T(SM)\to \mathscr{F} $ be the natural projection. Then the tangent bundle $T(SM)$ admits a splitting
\begin{equation}\label{splitting}
	T(SM)=\mathfrak{l}\oplus J\mathscr{F}\oplus \mathscr{F}=:H(SM)\oplus \mathscr{F}.
\end{equation}
where  $\mathfrak{l}$ is the trivial bundle generated by $\mathbf{G}$. The cotangent bundle $T^*(SM)$ has the corresponding splitting.

The Berwald connection $\nabla^{\rm Be}$ on the horizontal subbundle $H(SM)$ is a linear connection without torsion. The curvature $R^{\rm Be}=(\nabla^{\rm Be})^2$ of the Berwald connection is an ${\rm End}(H(SM))$-valued two form on $SM$. 
With respect to the natural splitting (\ref{splitting}), the curvature $R^{\rm Be}$ is divided into two parts
\begin{equation}
	R^{\rm Be}=\widetilde{R}+\widetilde{P},
\end{equation} 
where $\widetilde{R}$ contains pure \textit{hh}-forms, $\widetilde{P}$ contains \textit{hv}-forms.
The \textit{mean Berwald curvature} or the $\mathbf{E}$-curvature is defined as the trace of the Berwald curvature $\widetilde{P}$.

Let $\{\mathbf{e}_1,\ldots,\mathbf{e}_n,\mathbf{e}_{n+1},\ldots,\mathbf{e}_{2n-1}\}$ be a local adapted orthonormal frame with respect to the Sasaki-type Riemannian metric
\begin{align}\label{sasaki metric}
	g^{T(SM)}=\sum_{i}\omega^i\otimes\omega^i+\sum_{\bar{\alpha}}\omega^{\bar{\alpha}}\otimes\omega^{\bar{\alpha}}=:g+\dot{g}
\end{align}
on $SM$, where $\{\mathbf{e}_1,\ldots,\mathbf{e}_n\}$ spans $H(SM)$, $\mathbf{e}_n=\mathbf{G}$  and $\mathbf{e}_{\bar{\alpha}}=-J\mathbf{e}_\alpha$. Let $\theta=\left\{\omega^1,\ldots,\omega^n,\omega^{n+1},\ldots,\omega^{2n-1}\right\}$ be the dual frame, then $\omega^n=\omega$ is the Hilbert form. 

Under the adapted dual frame, the mean Berwald curvature has the following expression 
\begin{equation}\label{E}
	\mathbf{E}:=\tr \tilde{P}=:E_{\gamma\mu}\omega^\gamma\wedge\omega^{\bar{\mu}}.
\end{equation}
Using the almost complex structure $J$ and the $\mathbf{E}$-curvature (\ref{E}), we define the following symmetric vertical tensor
\begin{equation}
	\dot{\mathbf{E}}:=E_{\gamma\mu}\omega^{\bar{\gamma}}\otimes\omega^{\bar{\mu}}.
\end{equation} 
The trace of $\dot{\mathbf{E}}$ with respect to the vertical metric $\dot{g}$ is called the \textit{Berwald scalar curvature}
\begin{equation}
	\mathsf{e}:={\rm tr}_{\dot{g}}\dot{\mathbf{E}}.
\end{equation}
In this study, a Finsler manifold is called to have \textit{isotropic mean Berwald curvature} if the following condition is satisfied
\begin{equation}\label{iso E}
	\dot{\mathbf{E}}=\frac{\mathsf{e}}{n-1}\dot{g}.
\end{equation}
One notices that the Berwald scalar curvature $\mathsf{e}$ is a smooth function on $SM$ for a Finsler manifold. In the previous papers \cite{Li3,Lizhang}, a Finsler manifold $(M,F)$ is called to have\textit{ isotropic Berwald scalar curvature} if $\mathsf{e}\in C^{\infty}(M)$. 

The main result of this short note is a Schur type lemma for the mean Berwald curvature.
\begin{theo}\label{th 1}
	Let $(M,F)$ be a Finsler manifold and $\dim M=n\geq3$. If the mean Berwald curvature is isotropic, then the Berwald scalar curvature is isotropic.
\end{theo}
\begin{remark}
	In literature \cite{ChernShen,SS}, a Finsler manifold is called to have isotropic mean Berwald curvature if both the conditions (\ref{iso E}) and  $\mathsf{e}\in C^{\infty}(M)$ are satisfied. By Theorem \ref{th 1}, one notices that the assumption about the isotropic Berwald scalar curvature is unnecessary.  
\end{remark}

Combining the results in \cite{Li3} and Theorem \ref{th 1}, we obtain the following corollary.
\begin{mcor}\label{cor 1}
	Let $(M,F)$ be a Finsler manifold and $\dim M\geq3$. The following conditions are equivalent:
	\begin{enumerate}
		\item[(1).]  $F$ has isotropic mean Berwald curvature, i.e., the condition (\ref{iso E}) holds;
		\item[(2).]  the Berwald scalar curvature is isotropic, i.e., $\mathsf{e}\in C^{\infty}(M)$;
		\item[(3).]  $F$ has weakly isotropic $S$-curvature, and
		\begin{equation*}
			S=\frac{1}{n-1}\mathsf{e}+\pi^*\xi(\mathbf{G}),
		\end{equation*}
		where $\mathsf{e}\in C^{\infty}(M)$ and $\xi\in\Omega^1(M)$. 
	\end{enumerate}
\end{mcor}

\begin{remark}
	In the recent paper \cite{Cra3}, the author proves an attractive property about the the S-curvature based on his formula involved in the derivative of the coefficient of the Busemann-Hausdorff volume form. 
	
	To emphasize this important result, we would like to state it alone as follows: a Finsler structure $F$ has weakly isotropic $S$-curvature, i.e., (3) in Corollary \ref{cor 1} holds if and only if  the following statement is true:
	\begin{enumerate}
		\item[(4).] the $S$-curvature satisfies
	\begin{equation*}
		S=\frac{1}{n-1}\mathsf{e}+\pi^*df(\mathbf{G}),
	\end{equation*}
	where $\mathsf{e},f\in C^{\infty}(M)$. If the volume form on $M$ in the definition of the distortion is the Bussmann-Hausdorff one, then $f=0$ and $S$ is isotropic.
	\end{enumerate}
\end{remark}

\medskip

In this paper we adopt the index range and notations:
$$1\leq i,j,k,\ldots\leq n, \quad 1\leq \alpha,\beta,\gamma,\ldots\leq n-1, \quad 1\leq A,B,C,\ldots\leq 2n-1,$$

and $\bar{\alpha}:=n+\alpha,\bar{\beta}:=n+\beta,\bar{\gamma}:=n+\gamma,\ldots.$

\medskip

{\bf Acknowledgements.}  The author
would like to thanks Professor Huitao Feng for his consistent support and encouragement. 

\section{The geometry of the fibres of the sphere bundle of a Finsler manifold}

Let $(M,F)$ be a $n$-dimensional Finsler manifold. 
Let $$\nabla^{\rm Ch}:\Gamma(H(SM))\to\Omega^1(SM;H(SM))$$ be the Chern connection, which can be extended to a map
\begin{align*}
	\nabla^{\rm Ch}:\Omega^*(SM;H(SM))\rightarrow \Omega^{*+1}(SM;H(SM)),
\end{align*}
where $\Omega^*(SM;H(SM)):=\Gamma(\Lambda^*(T^*SM)\otimes H(SM))$ denotes the horizontal valued differential forms on $SM$.
It is well known that the symmetrization of Chern connection  $\hat{\nabla}^{\rm Ch}$ is the Cartan connection.
The difference between $\hat{\nabla}^{\rm Ch}$ and $\nabla^{\rm Ch}$ will be referred as the Cartan endomorphism,
\begin{align*}
	H= \hat{\nabla}^{\rm Ch}-\nabla^{\rm Ch}\quad\in\Omega^1(SM,{\rm
		End}(H(SM))).
\end{align*}
Set $H=H_{ij}\omega^j\otimes\mathbf{e}_i$. By Lemma 3 and Lemma 4 in \cite{FL}, $H_{ij}=H_{ji}=H_{ij\gamma}\omega^{\bar{\gamma}}$ has the following form under natural coordinate systems
\begin{align*}
	H_{ij\gamma}=-A_{pqk}u_i^pu_j^qu_{\gamma}^k,
\end{align*}
where $A_{ijk}=\frac{1}{4}F[F^2]_{y^iy^jy^k}$ and $u_i^j$ are the transformation matrix from adapted orthonormal frames to natural frames.


Let $\bm{\omega}=(\omega_j^i)$ be the connection matrix of the Chern
connection with respect to the local adapted orthonormal frame field,
i.e.,
\begin{align*}
	\D{\rm Ch}\mathbf{e}_i=\omega_i^j\otimes\mathbf{e}_j.
\end{align*}

\begin{lem}[\cite{BaoChernShen,ChernShen,FL}]\label{sturcture eq}
	The connection matrix $\bm{\omega}=(\omega_j^i)$ of $\nabla^{\rm Ch}$ is determined by the following structure equations,
	\begin{equation}\left\{
		\begin{aligned}
			&d\vartheta=-\bm{\omega}\wedge\vartheta,\\
			&\bm{\omega}+\bm{\omega}^t=-2H,
		\end{aligned}\right.\label{Chern connection structure eq. matrix}
	\end{equation}
	where $\vartheta=(\omega^1,\ldots,\omega^{n})^t$. Furthermore,
	$$\omega_{\alpha}^{n}=-\omega^{\alpha}_{n}=\omega^{\bar{\alpha}},\quad{\rm and}\quad \omega^{n}_{n}=0.$$
\end{lem}

In \cite{FL}, the authors proved that the Chern connection is just the Bott connection on $H(SM)$ in the theory of foliation (c.f. \cite{Zhang}).

Let $R^{\rm Ch}=\left(\nabla^{\rm Ch}\right)^2$  be the curvature of $\nabla^{\rm Ch}$.
Let $\Omega=\left(\Omega_j^i\right)$ be the curvature forms of $R^{\rm Ch}$. From the torsion freeness, the curvature form has no pure vertical differential form
\begin{align}\label{curvature chern connection}
	\Omega_j^i:=d\omega_j^i-\omega_j^k\wedge\omega_k^i=:\frac{1}{2}R_{j~kl}^{~i}\omega^k\wedge\omega^l+P_{j~k\gamma}^{~i}\omega^k\wedge\omega^{\bar{\gamma}}.
\end{align}

The Sasaki metric (\ref{sasaki metric}) induces a Euclidean metric $\dot{g}$ on the vertical bundle $\mathscr{F}$.
As the restriction the Levi-Civita connection $\nabla^{T(SM)}$ on $\mathscr{F}$, $\nabla^{\mathscr{F}}:=p\nabla^{T(SM)}p$ is a Euclidean connection of the bundle $(\mathscr{F},\dot{g})$. It is clear that along each fibre of $S_xM$, $x\in M$, $\nabla^{\mathscr{F}}$ is just the Levi-Civita connection $\dot{\nabla}^{\mathscr{F}}$ of the Riemannian manifold $(S_xM, \dot{g}|_{T(S_xM)})$.

The following lemma can be found in \cite{Mo}. We present here a new proof.
\begin{lem}\label{lem nabla f}
	Let $\Theta=(\Theta_B^A)$ be the connection form of the Levi-Civita connection $\nabla^{T(SM)}$ of $g^{T(SM)}$ with respect to the adapted frame $\{\mathbf{e}_A\}$. Then the connection forms of 	$\nabla^{\mathscr{F}}$ are determined by
	\begin{equation*}
		\nabla^{\mathscr{F}}\mathbf{e}_{\bar{\beta}}=\Theta_{\bar{\beta}}^{\bar{\gamma}}\otimes\mathbf{e}_{\bar{\gamma}},
	\end{equation*}
	  and 
	\begin{equation}\label{connection forms F}
		\Theta_{\bar{\beta}}^{\bar{\gamma}}=\omega_{\beta}^\gamma+H_{\beta\gamma\alpha}\omega^{\bar{\alpha}},
	\end{equation}
	where $\omega_{\beta}^\gamma$ are the connection forms of the Chern connection and $H_{\beta\gamma\alpha}\omega^{\bar{\alpha}}$ are the coefficients of the Cartan endomorphism. 
\end{lem}

\begin{proof}
	By the definition of the Levi-Civita connection, Lemma \ref{sturcture eq} and (\ref{curvature chern connection}), we obtain
	\begin{equation*}
		\begin{split}
			\langle\nabla^{T(SM)}_{\mathbf{e}_{\bar{\alpha}}}\mathbf{e}_{\bar{\beta}},\mathbf{e}_{\bar{\gamma}}\rangle
			=&\frac{1}{2}(\mathscr{L}_{\mathbf{e}_{\bar{\beta}}}g^{T(SM)})(\mathbf{e}_{\bar{\alpha}},\mathbf{e}_{\bar{\gamma}})+\frac{1}{2}d\omega^{\bar{\beta}}(\mathbf{e}_{\bar{\alpha}},\mathbf{e}_{\bar{\gamma}})\\
			=&-\frac{1}{2}\left(\langle[\mathbf{e}_{\bar{\beta}},\mathbf{e}_{\bar{\alpha}}],\mathbf{e}_{\bar{\gamma}}\rangle+\langle[\mathbf{e}_{\bar{\beta}},\mathbf{e}_{\bar{\gamma}}],\mathbf{e}_{\bar{\alpha}}\rangle+\langle[\mathbf{e}_{\bar{\alpha}},\mathbf{e}_{\bar{\gamma}}],\mathbf{e}_{\bar{\beta}}\rangle\right)\\
			=&-\frac{1}{2}\left((\Omega_n^{\gamma}+\omega_n^i\wedge\omega_i^{\gamma})(\mathbf{e}_{\bar{\beta}},\mathbf{e}_{\bar{\alpha}})+(\Omega_n^{\alpha}+\omega_n^i\wedge\omega_i^{\alpha})(\mathbf{e}_{\bar{\beta}},\mathbf{e}_{\bar{\gamma}})\right.\\
			&\left.\quad\quad +(\Omega_n^{\beta}+\omega_n^i\wedge\omega_i^{\beta})(\mathbf{e}_{\bar{\alpha}},\mathbf{e}_{\bar{\gamma}})\right)\\
			=&-\frac{1}{2}\left(-\omega_\beta^\gamma(\mathbf{e}_{\bar{\alpha}})+\omega_\alpha^\gamma(\mathbf{e}_{\bar{\beta}})-\omega_\beta^\alpha(\mathbf{e}_{\bar{\gamma}})+\omega^\alpha_\gamma(\mathbf{e}_{\bar{\beta}})
			-\omega_\alpha^\beta(\mathbf{e}_{\bar{\gamma}})+\omega_\gamma^\beta(\mathbf{e}_{\bar{\alpha}})\right)\\
			=&\omega_\beta^\gamma(\mathbf{e}_{\bar{\alpha}})+H_{\beta\gamma\alpha},
		\end{split}
	\end{equation*}
	and
	\begin{equation*}
		\begin{split}
			\langle\nabla^{T(SM)}_{\mathbf{e}_i}\mathbf{e}_{\bar{\beta}},\mathbf{e}_{\bar{\gamma}}\rangle
			=&\frac{1}{2}(\mathscr{L}_{\mathbf{e}_{\bar{\beta}}}g^{T(SM)})(\mathbf{e}_i,\mathbf{e}_{\bar{\gamma}})+\frac{1}{2}d\omega^{\bar{\beta}}(\mathbf{e}_i,\mathbf{e}_{\bar{\gamma}})\\
			=&-\frac{1}{2}\left(\langle[\mathbf{e}_{\bar{\beta}},\mathbf{e}_i],\mathbf{e}_{\bar{\gamma}}\rangle+\langle[\mathbf{e}_{\bar{\beta}},\mathbf{e}_{\bar{\gamma}}],\mathbf{e}_i\rangle+\langle[\mathbf{e}_i,\mathbf{e}_{\bar{\gamma}}],\mathbf{e}_{\bar{\beta}}\rangle\right)\\
			=&-\frac{1}{2}\left((\Omega_n^{\gamma}+\omega_n^i\wedge\omega_i^{\gamma})(\mathbf{e}_{\bar{\beta}},\mathbf{e}_i)+\omega^j\wedge\omega_j^i(\mathbf{e}_{\bar{\beta}},\mathbf{e}_{\bar{\gamma}})
			+(\Omega_n^{\beta}+\omega_n^i\wedge\omega_i^{\beta})(\mathbf{e}_i,\mathbf{e}_{\bar{\gamma}})\right)\\
			=&-\frac{1}{2}\left(-P^{~\gamma}_{n~i\beta }-\omega_\beta^\gamma(\mathbf{e}_i)+P^{~\beta}_{n~i\gamma }+\omega_\gamma^\beta(\mathbf{e}_i)\right)\\
			=&\omega_\beta^\gamma(\mathbf{e}_i),
		\end{split}
	\end{equation*}
	where we denote $\langle\cdot,\cdot\rangle:=g^{T(SM)}(\cdot,\cdot)$. 
\end{proof}
One notices that the other parts of $\Theta$ can be obtained by a similar way and will be omitted here.

In Lemma \ref{lem nabla f}, we modify the Cartan endmorphism to a vertical tensor
\begin{equation}
	\dot{H}:=H_{\alpha\beta\gamma}\omega^{\bar{\alpha}}\otimes\omega^{\bar{\beta}}\otimes\omega^{\bar{\gamma}}.
\end{equation}
Furthermore, the covariant differential of $\dot{H}$ by the connection $\nabla^{\mathscr{F}}$ can be expressed as
\begin{equation}\label{dH}
\begin{split}
		\nabla^{\mathscr{F}}\dot{H}&=(dH_{\alpha\beta\gamma}-H_{\mu\beta\gamma}\Theta_{\bar{\alpha}}^{\bar{\mu}}-H_{\alpha\mu\gamma}\Theta_{\bar{\beta}}^{\bar{\mu}}-H_{\alpha\beta\mu}\Theta_{\bar{\gamma}}^{\bar{\mu}})\otimes\omega^{\bar{\alpha}}\otimes\omega^{\bar{\beta}}\otimes\omega^{\bar{\gamma}}\\
		&=:(H_{\alpha\beta\gamma|i}\omega^i+H_{\alpha\beta\gamma:\mu}\omega^{\bar{\mu}})\otimes\omega^{\bar{\alpha}}\otimes\omega^{\bar{\beta}}\otimes\omega^{\bar{\gamma}}.
\end{split}
\end{equation}
We will call that $H_{\alpha\beta\gamma|i}$ the horizontal derivatives of $\dot{H}$ by $\nabla^{\mathscr{F}}$ along $\mathbf{e}_i$. One finds that the horizontal derivatives of $\dot{H}$ by  $\nabla^{\mathscr{F}}$ and the horizontal derivatives of $H$ by  $\nabla^{\rm Ch}$ coincide.
Similarly, $H_{\alpha\beta\gamma:\mu}$ will be called the vertical derivatives of $\dot{H}$ by $\nabla^{\mathscr{F}}$ along $\mathbf{e}_{\mu}$.

Let $\left(\nabla^{\mathscr{F}}\right)^2$ be the curvature of $\nabla^{\mathscr{F}}$. Set
\begin{equation}\label{curvature def}
	\left(\nabla^{\mathscr{F}}\right)^2\mathbf{e}_{\bar{\beta}}=\Omega_{\bar{\beta}}^{\bar{\gamma}}\otimes\mathbf{e}_{\bar{\gamma}}.
\end{equation}
In the following lemma, we calculate the curvature forms of $\left(\nabla^{\mathscr{F}}\right)^2$.
\begin{lem}
	The curvature forms $\Omega_{\bar{\beta}}^{\bar{\gamma}}$ of the connection $\nabla^{\mathscr{F}}$ are given by
	\begin{equation}\label{curvature c}
	\begin{split}
		\Omega_{\bar{\beta}}^{\bar{\alpha}}=&\frac{1}{2}(\Omega_{\beta}^{\alpha}-\Omega^{\beta}_{\alpha})+\frac{1}{2}(H_{\beta\alpha\nu,\mu}-H_{\beta\alpha\mu,\nu})\omega^{\bar{\mu}}\wedge\omega^{\bar{\nu}}\\
	&+\frac{1}{2}(H_{\beta\gamma\mu}H_{\gamma\alpha\nu}-H_{\beta\gamma\nu}H_{\gamma\alpha\mu})\omega^{\bar{\mu}}\wedge\omega^{\bar{\nu}}-\frac{1}{2}(\delta_{\beta\mu}\delta_{\alpha\nu}-\delta_{\beta\nu}\delta_{\alpha\mu})\omega^{\bar{\mu}}\wedge\omega^{\bar{\nu}},
	\end{split}
	\end{equation}
where $\Omega_{\beta}^{\alpha}$ are the curvature forms of the Chern connection.
\end{lem}
\begin{proof}
	By the definition (\ref{curvature def}), the curvature forms of $\nabla^{\mathscr{F}}$ are given by 
	\begin{equation}\label{curvature c0}
		\Omega_{\bar{\beta}}^{\bar{\alpha}}=d\Theta_{\bar{\beta}}^{\bar{\alpha}}-\Theta_{\bar{\beta}}^{\bar{\gamma}}\wedge\Theta_{\bar{\gamma}}^{\bar{\alpha}}.
	\end{equation}
Substituting the connection forms (\ref{connection forms F}) into (\ref{curvature c0}), by Lemma \ref{Chern connection structure eq. matrix} and (\ref{curvature chern connection}), one obtains 
\begin{equation}\label{curvature c1}
	\begin{split}
		\Omega_{\bar{\beta}}^{\bar{\alpha}}
		=&d(\omega_{\beta}^{\alpha}+H_{\beta\alpha\mu}\omega^{\bar{\mu}})-(\omega_{\beta}^{\gamma}+H_{\beta\gamma\mu}\omega^{\bar{\mu}})\wedge(\omega_{\gamma}^{\alpha}+H_{\gamma\alpha\nu}\omega^{\bar{\nu}})\\
		=&d\omega_{\beta}^{\alpha}-\omega_{\beta}^{i}\wedge\omega_{i}^{\alpha}+\omega_{\beta}^{n}\wedge\omega_{n}^{\alpha}-(H_{\beta\gamma\mu}\omega^{\bar{\mu}})\wedge(H_{\gamma\alpha\nu}\omega^{\bar{\nu}})\\
		&+dH_{\beta\alpha\mu}\wedge\omega^{\bar{\mu}}-H_{\beta\alpha\gamma}(\Omega_n^{\gamma}+\omega_n^{\mu}\wedge\omega_{\mu}^{\gamma})-\omega_{\beta}^{\gamma}\wedge(H_{\gamma\alpha\nu}\omega^{\bar{\nu}})-(H_{\beta\gamma\mu}\omega^{\bar{\mu}})\wedge\omega_{\gamma}^{\alpha}\\
		=&\Omega_{\beta}^{\alpha}-H_{\beta\alpha\gamma}\Omega_n^{\gamma}-\omega^{\bar{\beta}}\wedge\omega^{\bar{\alpha}}-(H_{\beta\gamma\mu}\omega^{\bar{\mu}})\wedge(H_{\gamma\alpha\nu}\omega^{\bar{\nu}})\\
		&+dH_{\beta\alpha\mu}\wedge\omega^{\bar{\mu}}-H_{\beta\alpha\gamma}(\Theta_{\bar{\mu}}^{\bar{\gamma}}-H_{\mu\gamma\nu}\omega^{\bar{\nu}})\wedge\omega^{\bar{\mu}}\\
		&-(\Theta_{\bar{\beta}}^{\bar{\gamma}}-H_{\beta\gamma\mu}\omega^{\bar{\mu}})\wedge(H_{\gamma\alpha\nu}\omega^{\bar{\nu}})-(H_{\beta\gamma\mu}\omega^{\bar{\mu}})\wedge(\Theta_{\bar{\gamma}}^{\bar{\alpha}}-H_{\gamma\alpha\nu}\omega^{\bar{\nu}})\\
		=&\Omega_{\beta}^{\alpha}-H_{\beta\alpha\gamma}\Omega_n^{\gamma}-\omega^{\bar{\beta}}\wedge\omega^{\bar{\alpha}}+(H_{\beta\gamma\mu}\omega^{\bar{\mu}})\wedge(H_{\gamma\alpha\nu}\omega^{\bar{\nu}})\\
		&+dH_{\beta\alpha\mu}\wedge\omega^{\bar{\mu}}-H_{\beta\alpha\gamma}\Theta_{\bar{\mu}}^{\bar{\gamma}}\wedge\omega^{\bar{\mu}}-H_{\gamma\alpha\mu}\Theta_{\bar{\beta}}^{\bar{\gamma}}\wedge\omega^{\bar{\mu}}-H_{\beta\gamma\mu}\Theta_{\bar{\alpha}}^{\bar{\gamma}}\wedge\omega^{\bar{\mu}}.
	\end{split}
\end{equation}
The equations (\ref{dH}), (\ref{curvature c1}) and the following Bianchi identities (c.f. \cite{BaoChernShen,Mo})
\begin{equation*}
	\Omega_{\alpha}^{\beta}+\Omega^{\alpha}_{\beta}=2H_{\alpha\beta\gamma}\Omega^{\gamma}_n-2H_{\alpha\beta\gamma|i}\omega^i,
\end{equation*}
imply that
	\begin{equation*}
		\begin{split}
			\Omega_{\bar{\beta}}^{\bar{\alpha}}=&\Omega_{\beta}^{\alpha}-H_{\beta\alpha\gamma}\Omega_n^{\gamma}-\omega^{\bar{\beta}}\wedge\omega^{\bar{\alpha}}+(H_{\beta\gamma\mu}\omega^{\bar{\mu}})\wedge(H_{\gamma\alpha\nu}\omega^{\bar{\nu}})\\
			&+H_{\beta\alpha\mu|i}\omega^i\wedge\omega^{\bar{\mu}}+H_{\beta\alpha\mu:\nu}\omega^{\bar{\nu}}\wedge\omega^{\bar{\mu}}\\
			=&\frac{1}{2}(\Omega_{\beta}^{\alpha}-\Omega^{\beta}_{\alpha})+H_{\beta\alpha\mu:\nu}\omega^{\bar{\nu}}\wedge\omega^{\bar{\mu}}+(H_{\beta\gamma\mu}\omega^{\bar{\mu}})\wedge(H_{\gamma\alpha\nu}\omega^{\bar{\nu}})
			-\delta_{\beta\mu}\delta_{\alpha\nu}\omega^{\bar{\mu}}\wedge\omega^{\bar{\nu}}\\
			=&\frac{1}{2}(\Omega_{\beta}^{\alpha}-\Omega^{\beta}_{\alpha})+\frac{1}{2}(H_{\beta\alpha\nu:\mu}-H_{\beta\alpha\mu:\nu})\omega^{\bar{\mu}}\wedge\omega^{\bar{\nu}}\\
			&+\frac{1}{2}(H_{\beta\gamma\mu}H_{\gamma\alpha\nu}-H_{\beta\gamma\nu}H_{\gamma\alpha\mu})\omega^{\bar{\mu}}\wedge\omega^{\bar{\nu}}-\frac{1}{2}(\delta_{\beta\mu}\delta_{\alpha\nu}-\delta_{\beta\nu}\delta_{\alpha\mu})\omega^{\bar{\mu}}\wedge\omega^{\bar{\nu}}.
		\end{split}
	\end{equation*}
Hence the equation (\ref{curvature c}) is proved.
\end{proof}
The fibres of the sphere bundle $SM$ of a Finsler manifold $(M,F)$ are integral submanifolds of the distribution $\mathscr{F}$, which have the induced Riemannian metrics form the Sasaki metric $g^{T(SM)}$. 

For any $x\in M$, $(S_xM,\dot{g}|_{T(S_xM)})$ is a Riemannian manifold. The connection $\nabla^{\mathscr{F}}$ along $S_xM$ gives the corresponding Levi-Civita connection $\dot{\nabla}^{\mathscr{F}}$. 
Let $i_x:S_xM\to SM$ be the inclusion map. Set
\begin{equation*}
	\dot{\omega}^{\bar{\alpha}}:=i_x^*\omega^{\bar{\alpha}}, \quad \alpha=1,\ldots, n-1.
\end{equation*}
Then $\{\dot{\omega}^{n+1}, \ldots, \dot{\omega}^{2n-1}\}$ is the induced orthonormal dual frame of $(S_xM,\dot{g}|_{T(S_xM)})$.
Let $\dot{\Theta}_{\bar{\beta}}^{\bar{\alpha}}$ and  $\dot{\Omega}_{\bar{\beta}}^{\bar{\alpha}}$ be the connection forms and curvature forms of the Levi-Civita connection $\dot{\nabla}^{\mathscr{F}}$. The coefficients of the curvature forms are defined as
\begin{equation}\label{curvature lcd}
\dot{\Omega}_{\bar{\beta}}^{\bar{\alpha}}=:\frac{1}{2}\dot{R}_{\beta\alpha\mu\nu}\omega^{\bar{\mu}}\wedge\omega^{\bar{\nu}}.
\end{equation} 
\begin{lem}
	The structure equations of each fibre of the sphere bundle of a Finsler manifold are given by
	\begin{equation}\label{codazzi H}
	H_{\beta\alpha\nu:\mu}-H_{\beta\alpha\mu:\nu}=0,
	\end{equation}
and
\begin{equation}\label{curvature lc}
	\dot{R}_{\beta\alpha\mu\nu}=H_{\beta\gamma\mu}H_{\gamma\alpha\nu}-H_{\beta\gamma\nu}H_{\gamma\alpha\mu}-(\delta_{\beta\mu}\delta_{\alpha\nu}-\delta_{\beta\nu}\delta_{\alpha\mu}).
\end{equation}
\end{lem}
\begin{proof}
  It is clear that $\dot{\Theta}_{\bar{\beta}}^{\bar{\alpha}}=i_x^*\Theta_{\bar{\beta}}^{\bar{\alpha}}$. For any $x\in M$, the curvature forms of  $(S_xM,\dot{g}|_{T(S_xM)})$ are
 \begin{equation}\label{curvature lc0}
 \dot{\Omega}_{\bar{\beta}}^{\bar{\alpha}}=d(i_x^*\Theta_{\bar{\beta}}^{\bar{\alpha}})-(i_x^*\Theta_{\bar{\beta}}^{\bar{\gamma}})\wedge(i_x^*\Theta_{\bar{\gamma}}^{\bar{\alpha}})=i_x^*(d\Theta_{\bar{\beta}}^{\bar{\alpha}})-i_x^*(\Theta_{\bar{\beta}}^{\bar{\gamma}}\wedge\Theta_{\bar{\gamma}}^{\bar{\alpha}})=i_x^*\Omega_{\bar{\beta}}^{\bar{\alpha}}.
 \end{equation}
By (\ref{curvature chern connection}), (\ref{connection forms F}) and (\ref{curvature lc0}), one obtains
\begin{equation}\label{curvature lc1}
\begin{split}
	\dot{\Omega}_{\bar{\beta}}^{\bar{\alpha}}=&\frac{1}{2}(H_{\beta\gamma\mu}H_{\gamma\alpha\nu}-H_{\beta\gamma\nu}H_{\gamma\alpha\mu})\omega^{\bar{\mu}}\wedge\omega^{\bar{\nu}}-\frac{1}{2}(\delta_{\beta\mu}\delta_{\alpha\nu}-\delta_{\beta\nu}\delta_{\alpha\mu})\omega^{\bar{\mu}}\wedge\omega^{\bar{\nu}}\\
	&+\frac{1}{2}(H_{\beta\alpha\nu:\mu}-H_{\beta\alpha\mu:\nu})\omega^{\bar{\mu}}\wedge\omega^{\bar{\nu}}.
\end{split}
\end{equation}
By (\ref{curvature lcd}) and (\ref{curvature lc1}), the following equation holds
\begin{equation}\label{curvature lc2}
	\dot{R}_{\beta\alpha\mu\nu}=H_{\beta\gamma\mu}H_{\gamma\alpha\nu}-H_{\beta\gamma\nu}H_{\gamma\alpha\mu}-(\delta_{\beta\mu}\delta_{\alpha\nu}-\delta_{\beta\nu}\delta_{\alpha\mu})+(H_{\beta\alpha\nu:\mu}-H_{\beta\alpha\mu:\nu}).
\end{equation}
The structure equations (\ref{codazzi H}) and (\ref{curvature lc}) can be derived from (\ref{curvature lc2}) directly.
\end{proof}

\section{The proof of the Schur type lemma}

Let $dV_M=\sigma(x)dx^1\wedge\cdots\wedge dx^n$ be any volume form of $M$.
The distortion of $(M,F)$ related to $dV_M$ is defined by
$$\tau:=\ln\frac{\sqrt{\det{(g_{ij}(x,y))}}}{\sigma(x)}.$$
The derivative of $\tau$ along the vector field $\mathbf{G}$ will be denoted by $S:=\mathbf{G}(\tau)$ and called the S-curvature. In the previous studies \cite{Li3,Lizhang}, a vertical differential equation involved the S-curvature and the mean Berwald curvature plays an important role. Using notations in this paper, We restate the mentioned equation for convenience 
\begin{equation}\label{hessian S}
	S_{:\alpha:\beta}+H_{\alpha\beta\mu}S_{:\mu}+S\delta_{\alpha\beta}=E_{\alpha\beta}.
\end{equation}

In the following lemma, we present the Codazzi equation of $\dot{E}$ along the fibres of the sphere bundle of a Finsler manifold.
\begin{lem}
	The Codazzi equation of the mean Berwald curvature is given by
	\begin{equation}\label{codazzi E}
		E_{\alpha\beta:\gamma}-E_{\alpha\gamma:\beta}=H_{\alpha\beta\mu}E_{\mu\gamma}-H_{\alpha\gamma\mu}E_{\mu\beta}.
	\end{equation}
\end{lem}
\begin{proof}
	The covariant derivative of (\ref{hessian S}) along fibres by the Levi-Civita connection $\dot{\nabla}^{\mathscr{F}}$ shows that
	\begin{equation}\label{codazi 1}
		E_{\alpha\beta:\gamma}=S_{:\alpha:\beta:\gamma}+H_{\alpha\beta\mu:\gamma}S_{:\mu}+H_{\alpha\beta\mu}S_{:\mu:\gamma}+S_{:\gamma}\delta_{\alpha\beta}.
	\end{equation}
By (\ref{hessian S}), (\ref{codazzi H}) and (\ref{codazi 1}), one obtains
\begin{equation}\label{codazzi 2}
	\begin{split}
		 &E_{\alpha\beta:\gamma}-E_{\alpha\gamma:\beta}\\
		=&S_{:\alpha:\beta:\gamma}-S_{:\alpha:\gamma:\beta}+(H_{\alpha\beta\mu:\gamma}S_{:\mu}-H_{\alpha\gamma\mu:\beta}S_{:\mu})\\
		&+(H_{\alpha\beta\mu}S_{:\mu:\gamma}-H_{\alpha\gamma\mu}S_{:\mu:\beta})+(S_{:\gamma}\delta_{\alpha\beta}-S_{:\beta}\delta_{\alpha\gamma})\\
		=&S_{:\alpha:\beta:\gamma}-S_{:\alpha:\gamma:\beta}+(S_{:\gamma}\delta_{\alpha\beta}-S_{:\beta}\delta_{\alpha\gamma})\\
		&+[H_{\alpha\beta\mu}(E_{\mu\gamma}-H_{\mu\gamma\nu}S_{:\nu}-S\delta_{\mu\gamma})-H_{\alpha\gamma\mu}(E_{\mu\beta}-H_{\mu\beta\nu}S_{:\nu}-S\delta_{\mu\beta})]\\
		=&S_{:\alpha:\beta:\gamma}-S_{:\alpha:\gamma:\beta}+(S_{:\gamma}\delta_{\alpha\beta}-S_{:\beta}\delta_{\alpha\gamma})\\
		&+S_{:\nu}(H_{\alpha\gamma\mu}H_{\mu\beta\nu}-H_{\alpha\beta\mu}H_{\mu\gamma\nu})+(H_{\alpha\beta\mu}E_{\mu\gamma}-H_{\alpha\gamma\mu}E_{\mu\beta}).
	\end{split}
\end{equation}
By (\ref{curvature lc}), the Ricci identity holds
\begin{equation}\label{Ricci id}
	\begin{split}
		&S_{:\alpha:\beta:\gamma}-S_{:\alpha:\gamma:\beta}=S_{:\nu}\dot{R}_{\alpha\nu\beta\gamma}\\
		=&S_{:\nu}[(H_{\alpha\mu\beta}H_{\mu\nu\gamma}-H_{\alpha\mu\gamma}H_{\mu\nu\beta})-(\delta_{\alpha\beta}\delta_{\nu\gamma}-\delta_{\alpha\gamma}\delta_{\nu\beta})]\\
		=&S_{:\nu}(H_{\alpha\mu\beta}H_{\mu\nu\gamma}-H_{\alpha\mu\gamma}H_{\mu\nu\beta})-(S_{:\gamma}\delta_{\alpha\beta}-S_{:\beta}\delta_{\alpha\gamma}) 
	\end{split}
\end{equation} 
Plugging (\ref{Ricci id}) into (\ref{codazzi 2}), one has
\begin{equation*}
	\begin{split}
		&E_{\alpha\beta:\gamma}-E_{\alpha\gamma:\beta}\\
	   =&S_{:\nu}(H_{\alpha\mu\beta}H_{\mu\nu\gamma}-H_{\alpha\mu\gamma}H_{\mu\nu\beta})-(S_{:\gamma}\delta_{\alpha\beta}-S_{:\beta}\delta_{\alpha\gamma})\\
	   &+(S_{:\gamma}\delta_{\alpha\beta}-S_{:\beta}\delta_{\alpha\gamma})+S_{:\nu}(H_{\alpha\gamma\mu}H_{\mu\beta\nu}-H_{\alpha\beta\mu}H_{\mu\gamma\nu})+(H_{\alpha\beta\mu}E_{\mu\gamma}-H_{\alpha\gamma\mu}E_{\mu\beta})\\
	   =&H_{\alpha\beta\mu}E_{\mu\gamma}-H_{\alpha\gamma\mu}E_{\mu\beta}.
	\end{split}
\end{equation*}
The proof is complete.
\end{proof}
Now we present a proof of Theorem \ref{th 1}.
\begin{proof}[The proof of Theorem \ref{th 1}]
	Assume that the mean Berwald curvature satisfies (\ref{iso E}), or equivalently 
	\begin{equation}\label{iso E'}
		E_{\alpha\beta}=\frac{1}{n-1}\mathsf{e}\delta_{\alpha\beta}.
	\end{equation}
Thus
\begin{equation}\label{a}
	E_{\alpha\beta:\gamma}=\frac{1}{n-1}\mathsf{e}_{:\gamma}\delta_{\alpha\beta}.
\end{equation}
On the other hand, by (\ref{codazzi E}) and (\ref{iso E'}), one directly has
\begin{equation}\label{codazzi 3}
	E_{\alpha\beta:\gamma}-E_{\alpha\gamma:\beta}=\frac{\mathsf{e}}{n-1}(H_{\alpha\beta\mu}\delta_{\mu\gamma}-H_{\alpha\gamma\mu}\delta_{\mu\beta})=0.
\end{equation}
Plugging (\ref{a}) into (\ref{codazzi 3}) yields
\begin{equation*}
	0=\mathsf{e}_{:\gamma}\delta_{\alpha\beta}-\mathsf{e}_{:\beta}\delta_{\alpha\gamma}.
\end{equation*}
Hence
\begin{equation*}
	0=(n-2)\mathsf{e}_{:\gamma}.
\end{equation*}
Assume that $n>2$, we obtain $\mathsf{e}_{:\gamma}=0$. Therefore the Berwald scalar curvature $\mathsf{e}$ is locally constant along fibres. By the connectedness of the fibres of the sphere bundle, we conclude that $\mathsf{e}$ is constant along fibres and 
$\mathsf{e}\in C^{\infty}(M)$.
\end{proof}

\end{document}